\documentclass[oneside,a4paper,11pt]{amsart}
\usepackage{amssymb,amsmath,t1enc,amsthm,geometry,enumitem}
\usepackage[utf8]{inputenc} 
\usepackage[T1]{fontenc}
\usepackage[english]{babel} 
\usepackage{url}
\usepackage{xargs}  
\usepackage[colorinlistoftodos,prependcaption,textsize=tiny]{todonotes}
\usepackage[super]{nth}

\title[Tangent Lie algebra of a diffeomorphism group]{Tangent Lie
  algebra of a diffeomorphism group and application to holonomy theory}

\author[B.~Hubicska]{Balázs Hubicska}
\email{\url{hubicska.balazs@science.unideb.hu}}

\author[Z.~Muzsnay]{Zolt\'an Muzsnay} 
\email{\url{muzsnay@science.unideb.hu}}

\address{University of Debrecen, Institute of Mathematics, Pf.~400, Debrecen,
  4002, Hungary}



\newcommand{\halmazvonal}[2]{\left\{\,#1\mid #2\,\right\}}
\newcommand{\halmazpont}[2]{\left\{\,#1 : #2\,\right\}}

\newcommand{\Hol}{\mathcal{H}ol} 
\newcommand{\hol}{\mathfrak{hol}\hspace{1pt}} 
\newcommand{\ihol}{\mathfrak{hol}^*} 
\newcommand{\Holf}{\mathcal{H}ol_f(M)} 
\newcommand{\holf}{\mathfrak{hol}_f(M)} 
\newcommand{\R}{\mathbb R}
\newcommand{\N}{\mathbb N}
\newcommand*{\p}{\mathcal{P}}
\newcommand*{\I}{\mathcal{I}}
\newcommand*{\F}{\mathcal{F}}
\newcommand*{\G}{\mathcal G}
\newcommand*{\GL}{{\mathcal G}_L}
\newcommand*{\g}{\mathfrak g}
\newcommand*{\gl}{{\mathfrak g}_L}

\newcommand*{\X}[1]{{\mathfrak X}^\infty\!\left(#1\right)}
\newcommand{\ts}{\textsuperscript}
\newcommand{\tu}[1]{_{_{#1}}}
\newcommand{\TG}{\mathcal T_{\hspace{-1pt}o} \mathcal G}
\newcommand*{\curv}{\mathfrak{R}}

\newcommand{\diff}[1]{\mathcal{D}i\!f\!f^{\infty}(#1)}

\def\={\!=\!}

\theoremstyle{plain}
\newtheorem{theorem}{Theorem}[section]
\newtheorem{proposition}[theorem]{Proposition}
\newtheorem{lemma}[theorem]{Lemma}
\newtheorem{corollary}[theorem]{Corollary}

\theoremstyle{definition}
\newtheorem{definition}[theorem]{Definition}
 
\newtheorem{remark}[theorem]{Remark}

\begin{document}

\maketitle

\begin{abstract}
  In this paper we introduce the notion of tangent space $\TG$ of a (not
  necessary smooth) subgroup $\G$ of the diffeomorphism group $\diff M$ of
  a compact manifold $M$.  We prove that $\TG$ is a Lie subalgebra of the
  Lie algebra of smooth vector fields on $M$. The construction can be
  generalized to subgroups of any (finite or infinite dimensional) Lie
  groups.  The tangent Lie algebra $\TG$ introduced this way is a
  generalization of the classical Lie algebra in the smooth cases.  As a
  working example we discuss in detail the tangent structure of the
  holonomy group and fibered holonomy group of Finsler manifolds.
\end{abstract}

\section{Introduction}


Important geometric objects, structures or properties can often be
investigated through algebraic structures.  In many interesting cases, these
algebraic structures are groups, where the group operations are smooth maps.
Such groups became indispensable tools for modern geometry, analysis, and
theoretical physics.  Lie groups and diffeomorphism groups are the most
important examples for such structures.

Considering a Lie group $\G\tu{L}$, it is well known that most of the
important information about it is captured in its tangent object, the Lie
algebra $\g\tu{L}$.  Naturally, if $\G$ is a Lie subgroup of $\G\tu{L}$, then
its Lie algebra $\g$ is a Lie subalgebra of $\g\tu L$.  The Lie subalgebra
$\g\subset \g\tu{L}$ can be used to obtain information or eventually to
determine the subgroup $\G$.  In many relevant geometric situations, however,
this framework is not general enough because of two factor: Firstly,
$\G\tu{L}$ is not a (finite dimensional) Lie group but the (infinite
dimensional) diffeomorphism group $\diff M$ of some manifold $M$.  Secondly,
the subgroup $\G$ is not necessarily a Lie subgroup of $\diff M$.
Nevertheless, natural questions arise: \emph{can we introduce a tangential
  property and tangent objects to the subgroup $\G$ in this situation?  Does
  the set of tangent elements possess a special algebraic stricture?  Can this
  algebraic structure be used to get information about the subgroup and thus
  about geometric properties?}  In this paper we answer these questions.

We introduce the notion of tangent vector fields to a subgroup $\G$ of the
diffeomorphism group $\diff{M}$, where $M$ is a compact manifold.  Denoting
by $\TG$ the set of tangent vector fields to $\G$ at the identity, we prove
that $\TG$ is a Lie subalgebra of the Lie algebra of smooth vector fields
on $M$ (Theorem \ref{thm:erinto_algebra}).  It follows that subalgebras of
$\TG$ inherit the tangential properties, therefore the elements of a
subalgebra generated by vector fields tangent to the subgroup $\G$ are
tangent to $\G$ (Corollary \ref{cor:sub_algebra}).  This property can be
particularly interesting when the Lie bracket of two tangent vector fields
to $\G$ generates a new direction: the tangential property will be
satisfied in this new direction as well. As we show in Theorem
\ref{thm:expo}, the group generated by the exponential image of $\TG$ is a
subgroup of the closure of $\G$ in $\diff{M}$ which can give important
information about the group $\G$ itself, especially in the infinite
dimensional cases.

We note, that a similar tangential property was already introduced in
\cite[Definition 2]{Muzsnay_Nagy_2011}, but we also remark that the concept
had two major defects: the tangent property introduced in
\cite{Muzsnay_Nagy_2011} is not preserved under the bracket operation,
therefore in that approach it is not true that tangent vector fields to a
subgroup $\G$ generate a tangent Lie algebra to $\G$.  Secondly,
\cite{Muzsnay_Nagy_2011} was not able to guaranty the existence of the tangent
Lie algebra $\TG$ associated to $\G$.  We note that with our approach we are
able to overcome both deficiencies.

The main reason to investigate the tangent structure of a subgroup $\G$ of the
diffeomorphism group is that it can provide valuable information about the
group $\G$ itself.  This method can be very effective when $\G$ is for example
a symmetry group, the holonomy group, etc. We note that in many cases the
determination of $\TG$ or its subalgebras can be highly nontrivial, especially
in the infinite dimensional cases.  As working examples, we consider the
holonomy group and the fibered holonomy group of Finsler manifolds. The
holonomy group is the transformation group generated by parallel translations
with respect to the canonical connection along closed curves. For Riemannian
manifolds it has been extensively studied and now the complete classification
is known \cite{Borel_Lichn_1952, Berger_1955, Bryant_2000, Joyce_2000}.  In
particular, it is well known, that the holonomy group of a simply connected
Riemannian manifold is a closed Lie subgroup of the special orthogonal group
$SO(n)$.  Despite the analogues in the construction, Finslerian holonomy
groups can be much more complex and up to now, we do not know much about them:
For special spaces the holonomy can be a finite dimensional Lie group (see
\cite{Szabo_1981} and \cite{Kozma_2000}), but recent results show that there
are Finsler manifolds with infinite dimensional holonomy group
\cite{Muzsnay_Nagy_2012, Muzsnay_Nagy_2014, Muzsnay_Nagy_max_2015}.  These
latter results show the difficulties: on cannot use the well understood
principal bundle machinery in the investigation because the structure group
should be infinite dimensional.  In \cite{Michor_1991} P.~Michor proposed a
general setting for the study of infinite dimensional holonomy groups and
holonomy algebras which was the motivation for Z.~Muzsnay and P.T.~Nagy to
start investigating the tangent objects to a subgroup of the diffeomorphism
group \cite{Muzsnay_Nagy_2011}.  In this paper we are able to step forward:
using the results of Chapter \ref{sec:tangen_algebra}.~we are able to
introduce the notion of holonomy algebra and fibered holonomy algebra for
Finslerian manifolds.  By improving the results of \cite{Muzsnay_Nagy_2011} we
also prove that the curvature and the infinitesimal holonomy algebras
(resp.~their restrictions) are Lie subalgebras of the fibered holonomy
algebras (resp.~the holonomy algebra). We are confident that in the future,
the tools described above can be used successfully in the investigation of
geometric structures in general and in the holonomy theory in particular.

\bigskip

\section{Introduction}

\label{sec:finsl}

In this chapter we introduce the basic notions and concepts of Finsler
geometry which are necessary to understand in Chapter
\ref{sec:fibred_holonomy_algebra} and Chapter \ref{sec:holonomy_algebra} the
nontrivial application of the theory discussing the tangent structure of a
subgroup of the diffeomorphism group.  These notions are not necessary to
understand Chapter \ref{sec:tangen_algebra}, therefore the reader who is not
particularly interested in these applications, can jump directly to the next
chapter.

\smallskip

In this paper, $M$ denotes a $C^{\infty}$-smooth $n$-dimensional
manifold, $\X{M}$ is the Lie algebra of $C^{\infty}$ vector fields and
$\diff{M}$ is the group of $C^{\infty}$ diffeomorphisms of $M$.  We will
denote by $TM$ the tangent manifold and by
$\widehat{T}M=TM\setminus\{0\}$ the slit tangent manifold. Local
coordinate charts $(U, x^i)$ on $M$ induce local coordinate charts
$(\pi^{-1}(U), (x^i, y^i))$ on $TM$, where $\pi: TM \to M$ is the
canonical projection. The \textit{vertical distribution}
$\mathcal{V} TM\subset TTM$ on $TM$ is given by
$\mathcal{V} TM= \mathrm{Ker}\, \pi_{*}$.
 
\subsection{Finsler manifold}
\label{sec:finsl_subsection} \
\\[1ex]
A \emph{Finsler manifold} is a pair $(M,\mathcal F)$, where the norm
$\mathcal F \colon TM \to \mathbb{R}_+$ is a positively 1-homogeneous
continuous function, which is smooth on $\hat T M$ and the symmetric bilinear
form
\begin{displaymath}
  g_{x,y} \colon (u,v)\ \mapsto \ g_{ij}(x, y)u^iv^j=\frac{1}{2}
  \frac{\partial^2 \mathcal F^2_x(y+su+tv)}{\partial s\,\partial t}\Big|_{t=s=0}
\end{displaymath}
is positive definite at every $y\in \hat T_xM$.  The \emph{indicatrix} $\I_xM$
at $x \in M$ is a hypersurface of $T_xM$ defined by
\begin{equation}
  \label{eq:2}
  \I_xM \! =\! \halmazpont{y \in T_xM}{\mathcal F(y) \! = \! 1}.
\end{equation}
\emph{Geodesics} of $(M, \mathcal F)$ are determined by a system of second
order ordinary differential equations
\begin{math}
  \ddot{x}^i + 2 G^i(x,\dot x)=0,
\end{math}
$i = 1,\dots,n$ in a local coordinate system $(x^i,y^i)$ of $TM$, where
$G^i(x,y)$ are determined by the formula
\begin{math}
  4 G^i= g^{il}\bigl(2\frac{\partial g_{jl}}{\partial x^k} -\frac{\partial
    g_{jk}}{\partial x^l} \bigr) \, y^jy^k.
\end{math}

\subsection{Parallel translation}
\label{sec:parallel_subsection} \
\\[1ex]
A vector field $X(t)= X^i(t) \frac{\partial}{\partial x^i}$ along a curve
$c\colon [0,1]\to M$ is called \textit{parallel} if $D_{\dot{c} } X(t)=0$
where the covariant derivative is defined as
\begin{equation}
  D_{\dot{c} } X(t)=  \left( \frac{d X^i (t)}{dt} 
    + G_j ^i (c_t, X(t)) \, \dot{c_t}^j  \right) 
  \frac{\partial}{\partial x^i},
\end{equation}
with $G_j ^i = \frac{\partial G ^i}{\partial y^j}$.  Clearly, for any
$X_0\!\in\!T_{c(0)}M$ there is a unique parallel vector field $X(t)$ along the
curve $c$ such that $X_0\!=\!X(0)$. Moreover, if $X(t)$ is a parallel vector
field along $c$, then $\lambda X (t)$ is also parallel along $c$ for any
$\lambda \ge 0$. Then the \emph{homogeneous (nonlinear) parallel translation}
along a curve $c(t)$
\begin{equation}
  \label{eq:parallel_1}
  \p^t_{c}:T_{c_0} M \to T_{c_t}M
\end{equation}
is defined by the positive homogeneous map $\p^t_c : X_0\mapsto X_t$ given by
the value $X_t=X(t)$ of the parallel vector field with initial value
$X(0) =X_0$. We remark that \eqref{eq:parallel_1} preserves the Finslerian
norm, therefore it can be considered as a map between the indicatrices
\begin{equation}
  \label{eq:parallel_2}
  \p^t_{c}:\I_{c_0} M \to \I_{c_t}M.
\end{equation}
Moreover, since the parallel translation is a homogeneous map,
\eqref{eq:parallel_1} and \eqref{eq:parallel_2} determine each other.

\medskip

\subsection{Holonomy}
\label{sec:finsl_hol} \
\\[1ex]
The \emph{holonomy group} $\Hol_p(M)$ of a Finsler manifold $(M, F)$ at a
point $p\in M$ is the group generated by parallel translations along
piece-wise differentiable closed curves starting at $p$. Considering the
parallel translation \eqref{eq:parallel_2} on the indicatrix, a holonomy
element is a diffeomorphism
\begin{math}
  \p_{c}:\I_p \to \I_p,
\end{math}
therefore the holonomy group $\Hol_p (M)\subset \diff{\I_p}$ is a subgroup of
the diffeomorphism group of the indicatrix $\I_p$.  

In the particular case, when $(M, \F)$ is a simply connected Riemann manifold,
the holonomy group is a closed Lie subgroup of the special orthogonal group
$SO(n)$.  Finslerian holonomy groups can, however, be much more complex: in
\cite{Muzsnay_Nagy_2012, Muzsnay_Nagy_2014, Muzsnay_Nagy_max_2015} one can
find examples of Finsler manifolds with infinite dimensional holonomy groups.
Until now it is not known if the Finsler holonomy groups are (finite of
infinite dimensional) Lie groups or not.

\medskip 

\subsection{Horizontal lift and curvature}
\label{sec:horizontal_lift} \
\\[1ex]
The parallel translation on a Finsler manifold can also be introduced by
considering the associated Ehresmann connection
(cf. \cite{Szilasi_Lovas_Kertesz_2014}): the horizontal distribution is
determined by the horizontal lift $T_xM\to T_{(x,y)}TM$ defined in the local
basis as
\begin{equation}
  \label{eq:lift2}
  \left(\frac{\partial}{\partial x^i} \right)^{\!\! h} =
  \frac{\partial}{\partial x^i} -G^k _i (x,y) \frac{\partial}{\partial y^k},
\end{equation}
where $y\in T_x M$.  Since the horizontal distribution is complementary to the
vertical distribution we have the decomposition
\begin{math}
  T_y TM = \mathcal{H}_y \oplus \mathcal{V}_y
\end{math}
with canonical projectors $h\colon TTM\to \mathcal{H}$ and
$v\colon TTM\to \mathcal{V}$.  The image $\mathcal{H} \subset TTM$ is the
\emph{horizontal distribution} of the manifold.  The horizontal lift of a
curve $c\colon [0,1]\to M$ with initial condition $X_0\in T_{c_0}M$ is a curve
$c^h\colon [0,1]\to TM$ such that $\pi \circ c^h \= c$,
$\frac{d c^h}{dt}\=(\frac{d c}{dt})^h$ and $c^h(0) \= X_0$. Then the parallel
translation can be geometrically obtained as
\begin{math}
  \p^t_{c}(X_0) = c^h(t). 
\end{math}
We remark that the horizontal lift $\varphi^h_t$ of the flow $\varphi_t$ of a
vector field $X \!\in \!\X {M}$ is the flow of the horizontal lift of the
vector field $X^h\! \in \!\X {TM}$. Therefore the parallel translation along
the integral curves of $X$ can be calculated in terms of the horizontal lift
of the flow:
\begin{equation}
  \label{eq:parallel_3}
  \p^t_{\varphi} = \varphi_t^h. 
\end{equation}
The horizontal distribution $\mathcal{H} TM$ is, in general,
non-integrable. The obstruction to its integrability is given by the
\emph{curvature tensor}
\begin{math}
   R=\frac{1}{2}[h,h]
\end{math}
which is the Nijenhuis torsion of the horizontal projector \cite{Grifone_1972}.

\bigskip

\section{Tangent Lie algebra of a subgroup of the diffeomorphism group}
\label{sec:tangen_algebra}

\medskip

In this paragraph we investigate the tangential property and tangential
structure of subgroups of the diffeomorphism group.  Let $\G$ be a subgroup
of $\diff M$ where $M$ is a compact differentiable manifold.  We do not
suppose any special property on $\G$, in particular, we do not suppose that
$\G$ is a Lie subgroup of $\diff{M}$. Questions that we consider: can we
introduce a tangential property and tangent object to the subgroup $\G$?
Does the set of tangent elements possess a special algebraic structure?
Can this algebraic structure be used to get information about the subgroup?
In this paragraph, we answer all these questions.

A smooth curve $c\colon I \to M$ on the manifold $M$ has a
$(k\!-\!1)\ts{st}$-order singularity at $t=0$, if its derivatives vanish up to
order $k\!-\!1$, ($k\geq 0$).  It is well known that if a curve $c$ has a
$(k\!-\!1)\ts{st}$-order singularity at $0\in \R$ then its $k\ts{th}$ order
derivative $c^{(k)}(0)=X_p$ is a tangent vector at $p=c(0)$. In that case, the
curve $c$ is called a \emph{$k\ts{th}$-order integral curve} of the vector
$X_p\in T_pM$.  Extending this concept to vector fields, we can introduce the
following
\begin{definition}
  \label{def:int_curve}
  A $C^\infty-$smooth curve in the diffeomorphism group
  \begin{math}
    \varphi\colon I \to \diff{M},
  \end{math}
  $t\to \varphi_t$ is called an \emph{integral curve of the vector field}
  $X\in \X M$ if
  \begin{itemize}[topsep=2pt, partopsep=0pt,leftmargin=25pt]
    \setlength{\itemsep}{2pt} \setlength{\parskip}{0pt}
    \setlength{\parsep}{0pt}
  \item [\emph{(1)}] $\varphi_{0}=id_M$,
  \item [\emph{(2)}] there exists $k\in \N$ such that for any point $p\in M$
    the curve $t \to \varphi_t(p)$ is a $k\ts{th}$-order integral curve of
    \begin{math}
      X(p)\in T_{p}M.
    \end{math}
  \end{itemize}
  This $k\in \N$ is called the \emph{order} of the integral curve $\varphi_t$
  of the vector field $X$.
\end{definition}
\noindent
In particular, the flow $\varphi^X_t$ of
$X\in \X M$ is a $1\ts{st}$-order integral curve of $X$. Moreover, if $k>1$
and $t\to \varphi_t$ is a $k\ts{th}-$order integral curve of the vector field
$X$ then we have
\begin{equation}
  \label{eq:1}
  \varphi_0=id_M,\qquad
  \frac{\partial \varphi_t}{\partial t} \Big|_{t=0}\!\!\!\!\!=0,
  \quad \dots  \quad  
  \frac{\partial^{k-1} \varphi_t}{\partial t^{k-1}} \Big|_{t=0} \!\!\!\!\!= 0,
  \qquad 
  \frac{\partial^{k} \varphi_t}{ \partial t^{k}} \Big|_{t=0} \!\!\!\!\!=X.
\end{equation}
Let $\G\!\subset\!\diff{M}$ be an arbitrary subgroup of the
diffeomorphism group $\diff{M}$.  Using the terminology of Definition
\ref{def:int_curve} we introduce the following
\begin{definition}
  \label{def:tang_algebra}
  A vector field $X \in \X{M}$ is called \emph{tangent} to a subgroup
  $\G\!\subset\!\diff{M}$ of the diffeomorphism group if there exists an
  integral curve of $X$ in $\G$. The set of tangent vector fields of $\G$ is
  denoted by $\TG$.
\end{definition}
\noindent
\begin{remark}
  We have $X \in \TG$ if and only if there exists a $C^\infty-$smooth curve
  \begin{math}
    \varphi\colon I \to Di\!f\!\!f^{\infty}(M)
  \end{math}
  such that
  \begin{itemize}[topsep=2pt, partopsep=0pt,leftmargin=38pt]
    \setlength{\itemsep}{2pt} \setlength{\parskip}{0pt}
    \setlength{\parsep}{0pt}
  \item [(\emph{1})\ ] $\varphi_{t}\in \G$,
  \item [(\emph{2})\ ] $\varphi_{0}=id_M$,
  \item [(\emph{3})\ ] there exists $k\in \N$ such that equation \eqref{eq:1}
    is satisfied.
  \end{itemize}
\end{remark}
\noindent
One can observe that in Definition \ref{def:tang_algebra} we do not suppose
that $\G$ is a Lie subgroup of $\diff{M}$. Indeed, we use the differential
structure of the later to formulate the smoothness condition on the curve in
$\G$.  Nevertheless, we have the following

\begin{theorem}
  \label{thm:erinto_algebra}
  If $\G$ is a subgroup of $\diff{M}$, then $\TG$ is a Lie subalgebra of
  $\X M$.
\end{theorem}

\begin{proof}
  To prove the theorem, we have to show that
  \begin{subequations}
    \label{eq:sub_algebra}
    \begin{alignat}{4}
      \label{eq:sub_algebra_1}
      X,Y&\in \TG & & \quad \Rightarrow \quad& [X,Y]&\in \TG,
      \\
      \label{eq:sub_algebra_2}
      X,Y&\in \TG & & \quad \Rightarrow \quad& X+Y&\in \TG,
      \\
      \label{eq:sub_algebra_3}
      \lambda \in \R, \ X&\in \TG & & \quad \Rightarrow \quad& \lambda X &\in
      \TG.
    \end{alignat}
  \end{subequations}
  Indeed, let $X,Y\in \TG$, that is $X,Y\in \X{M}$ tangent to $G$. According
  to Definition \ref{def:int_curve}, there exist $k, l\in \mathbb{N}$ such
  that $\varphi_{t}, \psi _{t} \in \G$ are integral curves of $X$ and $Y$
  respectively. Let us suppose that $\varphi_{t}$ is a $k\ts{th}-$order
  integral curve of $X$ and $\psi_{t}$ is an $l\ts{th}$-order integral curve
  of $Y$ $(k,l\geq 1)$. Then
  \begin{equation}
    \label{eq:int_X}
    \varphi_0=id_M, \qquad   
    \left\{
      \frac{\partial^{i} \varphi_t}{\partial t^{i}} \Big|_{t=0} \!\!\!= 0
    \right\}_{1\leq i<k} \qquad  
    \frac{\partial^{k} \varphi_t}{ \partial t^{k}} \Big|_{t=0} =X,
  \end{equation}
  and
  \begin{equation}
    \label{eq:int_Y}
    \psi_0=id_M, \qquad   
    \left\{
      \frac{\partial^{j} \psi_t}{\partial t^{j}} \Big|_{t=0} \!\!\!\!\!\!= 0
    \right\}_{1\leq j<l}\quad  
    \frac{\partial^{l} \varphi_t}{ \partial t^{l}} \Big|_{t=0} =Y.
  \end{equation}

  \bigskip

  \noindent
  $\bullet$ \emph{Proof of} \eqref{eq:sub_algebra_1}.  The computation is
  similar to that of \cite{Mauhart_Michor_1992}: Considering the group
  theoretical commutator
  \begin{equation}
    \label{eq:commutator}
    \left[ \varphi _t, \psi _s \right]:=\varphi^{-1}_t\circ \psi^{-1}_s
    \circ \varphi _t\circ \psi _s,
  \end{equation}
  we get a two-parameter family of diffeomorphisms such that if one of the
  parameters $s$ or $t$ is zero then \eqref{eq:commutator} is the identity
  transformation.  From \eqref{eq:int_X} and \eqref{eq:int_Y} we also know
  that the first, potentially nonzero derivative is the $(k+l)\ts{th}$ order
  mixed derivative:
  \begin{equation}
    \label{eq:der_com_1}
    \begin{aligned}
      \frac{\partial ^{(k+l)} \left[ \varphi _t, \psi _s \right]}{\partial
        t^k \partial s^l} \Big| _{(0,0)} (p)
      & =\frac{\partial ^l}{\partial s^l} \Big| _{s=0} \left( \frac{\partial
          ^k \left(\varphi ^{-1} _s \circ \psi ^{-1} _t \circ \varphi _s \circ
            \psi _t (p)\right)}{\partial t^k} \Big| _{t=0} \right)
      \\
      & =\frac{\partial ^l}{\partial s^l} \Big| _{s=0} \left( d \left( \varphi
          ^{-1} _s \right)_{\varphi _s (p)} \circ \frac{\partial ^k \psi ^{-1}
          _t}{\partial t^k}\Big| _{t=0} \left( \psi _s (p)\right) \right),
    \end{aligned}
  \end{equation} 
  where $d \left( \varphi ^{-1} _s \right)_{\varphi _s (p)}$ denotes the
  tangent map (or Jacobi operator) of $\varphi ^{-1} _s$ at the point
  $\varphi _s (p)$. Since
  \begin{math}
    d \left( \varphi ^{-1} _{s=0} \right)_{\varphi _s (p)}=id,
  \end{math}
  the above formula can be written in the form
  \begin{equation}
    \label{eq:der_com_2}
    d\left(\frac{\partial ^l \varphi ^{-1} _s }{\partial s ^l} 
      \Big|_{s=0} \right)_p \frac{\partial ^k \psi ^{-1} _t (p)}{\partial
      t^k} \Big| _{t=0} + d\left(\frac{\partial ^k \psi ^{-1} _t }{\partial
        t ^k} \Big|_{t=0} \right) _p \frac{\partial ^l \varphi _s (p)}
    {\partial s^l} \Big| _{s=0}.
  \end{equation}
  From $\varphi _t \circ \varphi ^{-1} _t =id$ we get
  \begin{displaymath}
    0= \frac{\partial^k}{\partial t^k}\Big|_{t=0} \left( \varphi _t 
      \circ \varphi ^{-1} _t \right)  = X 
    + \frac{\partial^k ( \varphi ^{-1} _t )
    }{\partial t^k}\Big| _{t=0}
  \end{displaymath}
  which yields
  \begin{equation}
    \label{eq:varphi_inv}
    \frac{\partial^k(\varphi ^{-1} _t)}{\partial t^k} \big|_{t=0}= -X.
  \end{equation}
  Therefore we get that \eqref{eq:der_com_2} can be written as
  \begin{equation}
    \label{eq:der_com_3}
    d\left(\frac{\partial ^l \varphi_s }{\partial s ^l} \Big|_{s=0}
    \right)_p \frac{\partial ^k \psi_t (p)}{\partial t^k} \Big| _{t=0} 
    - d\left(\frac{\partial^k \psi_t}{\partial t ^k}\Big|_{t=0} \right)_p 
    \frac{\partial ^l \varphi _s (p)}{\partial s^l} \Big| _{s=0},
  \end{equation}
  which is the Lie bracket of the vector fields $X$ and $Y$, that is
  \begin{equation}
    \label{eq:der_com_4}
    \frac{\partial^{k+l} \left[ \varphi _t, \psi _s \right]}{\partial
      t^k \partial s^l} \Big|_{(0,0)} = \left[ Y, X\right]. 
  \end{equation}
  From \eqref{eq:der_com_4} we get that
  \begin{math}
    t \to \left[ \varphi_t, \psi_t \right]
  \end{math}
  is a $(k+l)\ts{th}$-order integral curve of $[X,Y]$ in $\G$. Therefore
  $[X,Y]\in \TG$ which proves \eqref{eq:sub_algebra_1}.
 
  \bigskip

    \noindent
    $\bullet$ \emph{Proof of} \eqref{eq:sub_algebra_2}.
    \\[1ex]
    For any $c_1,c_2,m_1,m_2\in \R$,
    $\phi_t=\varphi_{c_1 t^{m_1}} \circ \psi_{c_2 t^{m_2}}$ is a smooth curve
    in $\G$ with $\phi_0=\varphi_{0} \circ \psi_{0}=id_M$. Moreover, if $r=$
    denotes the least common multiple of $k$ and $l$ and
    \begin{displaymath}
      m_1=r/k, \quad m_2=r/l, \quad
      c_1= \bigl(m_1^k(r-k)! \bigr)^{-1/r}, \quad
      c_2= \bigl(m_2^l(r-l)!\bigr)^{-1/r},   
    \end{displaymath}
    one gets
    \begin{equation}
      \frac{\partial^r\phi_t}{\partial t^r}\Big|_{t=0}
      =  \frac{\partial^r}{\partial t^r}\Big|_{t=0} \left( \varphi_{c_1 t^{m_1}} 
        \circ \psi_{c_2 t^{m_2}} \right)  = X+Y,
    \end{equation}
    showing that $\psi_t$ is an $r\ts{th}$-order integral curve of $X+Y$ in
    $\G$, therefore $X+Y$ is tangent to $\G$.
    
    \bigskip

    \noindent
    $\bullet$ \emph{Proof of} \eqref{eq:sub_algebra_3}.
    \\[1ex]
    It is clear that in the case when $\lambda \geq 0$, one can reparametrize
    the integral curve of $X$, and using that the lower order terms are zero,
    we get
    \begin{equation}
      \frac{\partial^k \varphi_{\!_{\sqrt[k]{\lambda} t}}}{\partial t^k}  \Big| _{t=0} 
      = \lambda X.
    \end{equation}
    In the case when $\lambda < 0$ one can use \eqref{eq:varphi_inv} and we
    get
    \begin{equation}
      \frac{\partial^k}{\partial t^k} \Big| _{t=0}  
      \left(\varphi^{-1}_{\!_{\sqrt[k]{|\lambda |} t}}  \right) = -|\lambda|X = \lambda X
    \end{equation}
    From \emph{(21)} and \emph{(22)} we get that $\lambda X$ is tangent to
    $G$, that is $\lambda X \in \TG$, and from \emph{11b)} and \emph{11c)} we
    get that any linear combinations of $X$ and $Y$ are in $\TG$.
  \end{proof}

  \smallskip

  Motivated by the results of Theorem \ref{thm:erinto_algebra} we propose the
  following

\begin{definition}
  $T_0\G$ is called the tangent Lie algebra of the subgroup
  $\G\subset \diff{M}$.
\end{definition}
As a direct consequence of Theorem \ref{thm:erinto_algebra} we get the
following

\begin{corollary}
  \label{cor:sub_algebra}
  Let $\G$ be a subgroup of $\diff{M}$ and ${\mathcal S}$ be a subset of
  $\X M$ such that the elements of ${\mathcal S}$ are tangent to $\G$. Then
  the Lie subalgebra
  \begin{math}
    \langle { \, \mathcal S \, } \rangle_{\mathcal Lie}
  \end{math}
  of $\X M$ generated by the elements of ${\mathcal S}$ is also tangent to
  $\G$, that is
  \begin{displaymath}
    {\mathcal S}\subset \TG \qquad \Rightarrow \qquad    
    \big \langle  \, {\mathcal S}  \,    \big \rangle_{Lie} \subset \TG.
  \end{displaymath}
\end{corollary}

\begin{remark}
  \label{rem:tangent_properties}
  Slightly different tangent properties of vector fields to a subgroup $\G$ of
  the diffeomorphism group were already introduced in
  \cite{Muzsnay_Nagy_2011}. We will refer to the property \cite[Definition
  2.]{Muzsnay_Nagy_2011} as the \emph{weak tangent property} and to
  \cite[Definition 4.]{Muzsnay_Nagy_2011} as the \emph{strong tangent
    property}.  Our language is justified by the following proposition which
  is clarifying the relationship between the tangent property introduced in
  Definition \ref{def:int_curve} and the tangent properties introduced in
  \cite{Muzsnay_Nagy_2011}:
\end{remark}

\begin{proposition}
  \label{prop:weak_and_strong_tangent}
  Let $\G$ be a subgroup of $\diff{M}$ and $X\in \X{M}$. Using the terminology
  of Remark \ref{rem:tangent_properties}:
  \begin{enumerate}[topsep=0pt, partopsep=10pt,leftmargin=25pt]
  \item [(i)] if $X$ is strongly tangent to $\G$, then $X\in \TG$.
  \item [(ii)] if $X\in \TG$, then it is weakly tangent to $\G$.
  \end{enumerate}
\end{proposition}
\begin{proof}
  $\emph{(i)}$ If $X\in \X{M}$ is a strongly tangent vector field to the
  subgroup $\G \subset \diff{M}$, there exists a $k$-parameter commutator like
  family of diffeomorphisms $\phi_{t_1\dots t_k}\in \G$ which is
  $C^\infty$-smooth in $\diff{M}$, $\phi_{t_1,\dots, t_k} = id_M$ whenever one
  of its parameters is zero and
  \begin{displaymath}
    X=    \frac{\partial^k \phi_{t_1\dots t_k}}{\partial t_1\! \dots
      \partial t_k} \Big|_{(0 \dots  0)} .
  \end{displaymath}
  Consequently, if we consider the map $t \to \varphi_t$ where
  $\varphi_t= \phi_{t,\ldots , t}$, we get a $1$-parameter family of
  diffeomorphisms which satisfies the conditions of Definition
  \ref{def:tang_algebra}. Therefore, the vector field $X$ is tangent to $\G$.

  To prove $\emph{(ii)}$, let us suppose that $\varphi_{t}$ is a $k\ts{th}$
  order integral curve of $X$. Then we have \eqref{eq:int_X} and one can write
  $\varphi_{t}(p)$ as
  \begin{equation}
    \varphi_{t}(p) = p + \tfrac{1}{k!} \,
    t^k \bigl(X(p) + \omega(p,t) \bigr) 
  \end{equation}
  where $\lim_{t\to 0}\omega (p,t)=0$. The reparametrization
  $t \to \psi_{t}:=\varphi_{k!\sqrt[k]{t}}$ gives a $C^1$-differentiable
  1-parameter family of diffeomorphism in $\diff M$ such that $\psi_{0}=id_M$
  and
  \begin{displaymath}
    \frac{\partial \psi_{t}}{\partial t} \Bigl|_{t=0}(p) 
    =  \frac{\partial \varphi_{k!\sqrt[k]{t}}}{\partial t} \Bigl|_{t=0}(p) = X(p),
  \end{displaymath}
  which proves \emph{(ii)}.
\end{proof}

\begin{remark}
  \label{rem:tangenty_motivation}
  One may wonder why to introduce a new tangent property when there are
  already two, the weak and the strong tangent property (using the terminology
  of Remark \ref{rem:tangent_properties}) introduced in the literature. As an
  answer we point out that, the concept in \cite{Muzsnay_Nagy_2011} has a
  major defect: the weak tangent property is not preserved under the bracket
  operation, therefore it is not true in general that weakly tangent vector
  fields to a subgroup $\G$ generate a weakly tangent Lie algebra to $\G$. To
  overcome this difficulty, the authors introduced the strongly tangent
  property but the strongly tangent property was not preserved under the
  linear combination. It follows that \cite{Muzsnay_Nagy_2011} and the
  succeeding papers using these techniques were not able to guaranty the
  existence of the tangent Lie algebra $\TG$ associated to $\G$.  With our
  approach we are able to overcome this major deficiency.
\end{remark}

The main feature of $\TG$ is that one can obtain information about the group
$\G$. Indeed, one has the following

\begin{theorem}
  \label{thm:expo}
  Let $\G$ be a subgroup of $\diff{M}$ and $\overline{\G}$ its topological
  closure with respect to the $C^\infty$ topology. Then the group generated by
  the exponential image of the tangent Lie algebra $\TG$ with respect to the
  exponential map $\exp\colon \X{M} \to \diff{M}$ is a subgroup of
  $\overline{\G}$.
\end{theorem}

\begin{proof}
  From the proof of Proposition \ref{prop:weak_and_strong_tangent} we know
  that for any element $X\in \TG$ there exists a $C^1$-differentiable
  $1$-parameter family $\{\psi_t \} \subset \G$ of diffeomorphisms of $M$ such
  that $\psi_0=id_M$ and
  \begin{math}
    X=\frac{\partial \psi_t}{\partial t}\big|_{t=0}.
  \end{math}
  Then, using the argument of \cite[Corollary 5.4, p.~84]{Omori_1997} on
  $\psi_t$ we get that
  \begin{displaymath}
    \psi^{n}  \left(\tfrac{t}{n} \right) =\psi
    \left(\tfrac{t}{n} \right)   \circ \cdots
    \circ \psi   \left(\tfrac{t}{n} \right)
  \end{displaymath}
  in $\G$, as a sequence of $\diff{M}$, converges uniformly in all derivatives
  to $\exp(t X)$.  It follows that
  \begin{displaymath}
    \halmazvonal{\exp(tX)}{t\in\mathbb R} \subset \overline{\G},
  \end{displaymath}
  for any $X\!\in\!\TG$. Therefore, one has 
  \begin{math}
    \exp\, (\TG) \subset \overline{\G}
  \end{math}
  and if we denote by $\big\langle \exp (\TG) \big\rangle$ the group generated
  by the exponential image of $\TG$ we get
  \begin{displaymath}
    \big\langle\exp(\TG) \big\rangle  \subset  \overline{\G},
  \end{displaymath}
  which proves Theorem \ref{thm:expo}.
\end{proof}
We note that, assuming the manifold M is compact, we could avoid technical
difficulties. Indeed, in this case, the diffeomorphism group $\diff{M}$ is
an (infinite dimensional) manifold and the exponential image of the flow of
vector fields exists everywhere on $M$. For a more general and deeper
discussion of the subject see \cite{Wojtynski_2003}.

\bigskip

The concept worked out in Definition \ref{def:tang_algebra} and Theorem
\ref{thm:erinto_algebra} can be adapted not only for subgroups of the
diffeomorphism group but for any subgroup of any (finite or infinite
dimensional) Lie group:

\begin{definition}
  \label{def:tang_algebra_lie}
  Let $\G_{L}$ be a Lie group, $e\in \GL$ is the identity element of $\GL$ and
  $\g_{L}:=T_e\,\GL$ the Lie algebra of $\GL$.  If $\G\subset \GL$ is a
  subgroup of $\G_L$, then $X\in \gl$ is called tangent to $\G$ if there exist
  a $C^\infty-$smooth curve
  \begin{math}
    \varphi\colon I \to \GL
  \end{math}
  such that
  \begin{itemize}[topsep=2pt, partopsep=0pt,leftmargin=38pt]
    \setlength{\itemsep}{2pt} \setlength{\parskip}{0pt}
    \setlength{\parsep}{0pt}
  \item [(\emph{1})\ ] $\varphi_{t}\in \G$,
  \item [(\emph{2})\ ] $\varphi_{0}=e$,
  \item [(\emph{3})\ ] there exists $k\in \N$ such that $t \to \varphi_{t}$ is
    a $k\ts{th}$ order integral curve of $X$.
  \end{itemize}
  The set of tangent vector of $\G$ is denoted by $\TG$.
\end{definition}

Then, adapting the proof of Theorem \ref{thm:erinto_algebra} and Theorem
\ref{thm:expo} we can get the following

\begin{theorem}
  \label{thm:erinto_algebra_lie}
  If $\G$ is a subgroup of a Lie group $\GL$, then $\TG$ is a Lie subalgebra
  of $\gl$. The group $\big\langle\exp(\TG) \big\rangle$ generated by the
  exponential image of $\TG$ with respect to the exponential map
  $\exp\colon \gl \to \GL$ is a subgroup of the topological closure
  $\overline{\G}$ of $\G$ in $\GL$.
\end{theorem}
It is clear that in the case when $\G$ is a Lie subgroup of $\GL$, then
$\TG=\g$ is just the usual Lie subalgebra of $\gl$ associated to the Lie
subgroup $\G$. Therefore Definition \ref{def:tang_algebra_lie} generalizes the
classical notion of the Lie subalgebra associated to a Lie subgroup.

\bigskip \bigskip

\section{An application: holonomy algebra}
 
\bigskip

The notion of the holonomy group was already introduced in Chapter
\ref{sec:finsl_hol}.  It is well known that in the particular case when
$(M, \F)$ is a Riemann manifold, the holonomy group is a compact Lie subgroup
of the orthogonal group $O(n)$ and its Lie algebra is a Lie subalgebra of
$\mathfrak o (n)$.  It is also clear that the holonomy group of a linear
connection is a subgroup of the linear group $GL(n)$ and its Lie algebra is a
Lie subalgebra of $\mathfrak {gl} (n)$.  However, the situation for a Finsler
manifold or in a more general context the holonomy of a homogeneous connection
can be much more complex. Examples show that in some cases the holonomy
group can not be a finite dimensional Lie group \cite{Muzsnay_Nagy_2012_hjm,
  Muzsnay_Nagy_2012, Muzsnay_Nagy_2014}.  Until now it is not known if the
Finsler holonomy groups are (finite of infinite dimensional) Lie groups or
not. Nevertheless, the theory developed in Chapter \ref{sec:tangen_algebra}
allows us to consider its tangent Lie algebra, the holonomy algebra.

\bigskip

\subsection{The fibered holonomy algebra and its Lie subalgebras}
\label{sec:fibred_holonomy_algebra} \

\bigskip

\noindent
Let $(M, \F)$ be a compact Finsler manifold. The notion of \emph{fibered
  holonomy group} $\Holf$ appeared in \cite{Muzsnay_Nagy_2011}: It is the
group generated by fiber preserving diffeomorphisms $\Phi$ of the
indicatrix bundle $(\I M, \pi ,M)$, such that for any $p\in M$ the
restriction $\Phi _p =\Phi|_{\I_p}$ is an element of the holonomy group
$\Hol_p(M)$.  It is obvious that
\begin{equation}
  \Holf \subset \diff {\I M},
\end{equation}
where $\Holf$ is a subgroup of the diffeomorphism group of the indicatrix
bundle. Until now it is not known whether or not $\Holf$ is a Lie subgroup of
$\diff {\I M}$.  The set of tangent vector fields to the group $\Holf$ denoted
as
\begin{equation}
  \label{eq:holf}
  \holf:=\mathcal T_0 \bigl(\Holf \bigr).
\end{equation}
\begin{definition}
  $\holf$ is called the \emph{fibered holonomy algebra} of the Finsler
  manifold $(M, \F)$.
\end{definition}
From Theorem \ref{thm:erinto_algebra} one can obtain the following
\begin{theorem}
  The fibered holonomy algebra $\holf$ is a Lie subalgebra of the Lie algebra
  of smooth vector fields $\X {\I M}$.
\end{theorem}
In the sequel we will investigate the two most important Lie subalgebras of
$\holf$ which can be introduced with the help of the curvature tensor (see
Paragraph \ref{sec:finsl_hol}) of a Finsler manifold: the curvature algebra
and the infinitesimal holonomy algebra.

\begin{definition}
  \label{def:curvature_algebra}
  A vector field $\xi\in \X{\I M}$ is called a \emph{curvature vector field}
  if there exist vector fields $X,Y \in \X{M}$ such that $\xi = R(X^h,Y^h)$.
  The Lie subalgebra $\curv$ of vector fields generated by curvature vector
  fields is called the \emph{curvature algebra}.
\end{definition}
It is easy to see that from the definition of the curvature tensor that a
curvature vector field can be calculated as
\begin{equation}
  \label{eq:R}
  \xi = R (X^h,Y^h)=\big[ X^h ,Y^h \big]-\big[ X ,Y \big]^h,
\end{equation}
and from the definition we have also $\curv \subset \X{\I M}$. Moreover, we
have the following

\begin{proposition}
  \label{prop:gorb-erinto-uj}\
  \begin{enumerate}
  \item The elements of the curvature algebra are tangent to the group $\Holf$.
  \item The curvature algebra $\curv$ is a Lie subalgebra of $\holf$.
  \end{enumerate}
\end{proposition}

To prove the first part of the proposition, we have to show that the curvature
vector fields are tangent to the fibered holonomy group $\Holf$, that is they
are elements of $\holf$.  Let $\xi \in \X{\I M}$ be a curvature vector field
and $X,Y\in \X{M}$ such that $\xi\= R(X^h,Y^h)$. We denote by $\varphi$ and
$\psi$ the integral curves of $X$ and $Y$ respectively. Define
\begin{displaymath}
  \alpha_{t,s}  := 
  \left\{
    \begin{aligned}
      & \varphi_s, \quad & 0 \leq & \ s \leq t,
      \\
      \psi_{s-t}&\varphi_t, \quad & t \leq &\ s \leq 2t,
      \\
      \varphi_{2t-s}\psi_t & \varphi_t, \quad & 2t \leq & \ s \leq 3t,
      \\
      \psi_{3t-s} \varphi_{-t} \psi_t & \varphi_t, \quad & 3t \leq & \ s \leq
      4t.
    \end{aligned}
  \right.
\end{displaymath}
and
\begin{displaymath}
  \beta_{t,s} := \psi_{-s}\varphi_{-s} \psi_s \varphi_s,  \qquad 0\leq s \leq t.
\end{displaymath}
Then, for every $p\in M$ and fixed $t$ the map
\begin{math}
  \alpha_{t}(p)\colon s \to \alpha_{t,s}(p)
\end{math}
and
\begin{math}
  \beta_{t}(p)\colon s \to \beta_{t,s}(p)
\end{math}
are parametrized curves: $\alpha_{t}(p)\colon s \to \alpha_{t,s}(p)$ is a (not
necessarily closed) parallelogram and $\beta_t(p)$ joins the endpoints of
$\alpha_{t}(p)$. Indeed, for every $p\in M$ and fixed $t$ the endpoint of
$\alpha_t(p)$ coincides with the endpoint of $\beta_t(p)$ and consequently
\label{item:zart_2} the curve $\alpha_t(p)\ast \beta_t ^{-1}(p)$ defined as
going along the curve $\alpha_t(p)$ then continuing along $\beta^{-1}_t(p)$
(which is the curve $\beta_t(p)$ with opposed orientation) is a closed curve
that starts and ends at $p\in M$. Let us consider
\begin{equation}
  \label{eq:h}
  h_{t,p} := \p_{\alpha_t(p)\ast \beta_t ^{-1}(p)}
  = \p_{\alpha_t(p)} \circ \p_{\beta_t (p)} ^{-1} ,
\end{equation}
the parallel translation along
\begin{math}
  \alpha_t(p)\ast \beta_t ^{-1}(p).
\end{math}
We have the following

\begin{lemma} 
  \label{lemma:h_t_p}
  For any $p\in M$
  \begin{itemize}[topsep=2pt, partopsep=0pt,leftmargin=24pt]
    \setlength{\itemsep}{2pt} \setlength{\parskip}{0pt}
    \setlength{\parsep}{0pt}
  \item [(1)] $h_{t,p} \in \Hol_p(M)$,
  \item [(2)] $t\to h_{t,p}$ is a second order integral curve of the vector
    field $\xi_p:=\xi\big|_{\I_p}$ $\big(\xi_p\in \X{\I_p}\big)$.
  \end{itemize}
\end{lemma}

\begin{proof}
  Indeed, for every $p\in M$ and sufficiently small $t$ the curve
  $\alpha_t(p)\ast \beta_t ^{-1}(p)$ is a closed loop starting and ending at
  $p$, therefore the parallel transport $h_{t,p}:\I_p\to\I_p$ is a holonomy
  transformation at $p$ and we get \emph{(1)} of the lemma.

  To show \emph{(2)} we first remark that $\alpha_0(p)$ and $\beta_0(p)$ are
  the trivial curves ($s\to \alpha_{0,s}(p) = \beta_{0,s}(p) \equiv p$),
  therefore the parallel translation along them is the identity transformation
  and
  \begin{equation}
    \label{eq:h_0}
    h_{0,p} = id_{\I_p}.
  \end{equation}
  On the other hand, as we have seen in Chapter \ref{sec:finsl}, the parallel
  transport along a curve is determined by the horizontal lift of the
  curve. Consequently, the parallel transport along the integral curves of the
  vector fields $X$ and $Y$ can be expressed with the flows of the horizontal
  lifts $X^h$ and $Y^h$.  Let us consider first the parallel transport along
  the curve $\alpha_t(p)$: the parallel transport of a vector $v\in \I_p$
  along the curve $\alpha_t(p)$ is
  \begin{displaymath}
    \p_{\alpha_t(p)} (v)  = 
    \left\{
      \begin{aligned}
        & \varphi_s^{X^h}(v), \quad & 0 \leq & \ s \leq t,
        \\
        & \varphi_{s-t}^{Y^h}\varphi_t^{X^h}(v), \quad & t \leq &\ s \leq 2t,
        \\
        & \varphi^{X^h}_{-(s-2t)}\varphi_t^{Y^h} \varphi_t^{X^h}(v), \quad &
        2t \leq & \ s \leq 3t,
        \\
        & \varphi^{Y^h}_{-(s-3t)} \varphi^{X^h}_{-t} \varphi_t^{Y^h}
        \varphi_t^{X^h}(v), \quad & 3t \leq & \ s \leq 4t.
      \end{aligned}
    \right. 
  \end{displaymath}
  Therefore, $\p _{\alpha _t (p)}$ corresponds to the infinitesimal (not necessarily
  closed) parallelogram having as sides the integral curves of the horizontal
  lifts $X^h$ and $Y^h$. From the well known properties of the Lie brackets
  (see for example \cite[p.162]{Spivak_1979}) we get that
  \begin{equation}
    \label{eq:alpha}
      \frac{d}{dt}\Big|_{t=0} \p _{ \alpha_{t}} (v) = 0,
      \qquad \mathrm{and} \qquad
      \frac{d^2}{dt^2}\Big|_{t=0} \p_{ \alpha _{t}} (v) = 2\left[ X^h,Y^h
      \right]_v.
  \end{equation}
  On the other hand, the parallel transport of a vector
  $w\in \I_{\alpha _t(p)}$ along $\beta_t^{-1}(p)$ can be calculated with the
  help of it's horizontal lift
  \begin{math}
    \p_{ \beta _t ^{-1}}(w) = \p_{ \beta _t } ^{-1}(w) =((\beta)^h(t))^{-1}
    (w),
  \end{math}
  where by the definition of the horizontal lift
  $\pi \circ (\beta)^h (t)=\beta (t)$ and $(\beta^{-1})^h (0)=w$ are
  fulfilled. Since 
  \begin{math}
    \frac{d}{dt}\big|_{t=0} \beta_{t} (p) = 0,
  \end{math}
  and 
  \begin{math}
    \frac{d^2}{dt^2}\big|_{t=0} \beta_{t} (p) (v) = 2\left[ X,Y \right]_p, 
  \end{math}
  we obtain
   \begin{equation}
     \label{eq:beta}
     \frac{d}{dt} \Big|_{t=0} \p^{-1}_{\beta_t} = 0
     \qquad \mathrm{and} \qquad
     \frac{d^2}{dt^2} \Big| _{t=0} \p ^{-1} _{ \beta_t } (v)
     = -\big( 2\left[ X,Y\right]^h \big)_v,
   \end{equation}
   thus, from the two equations of \eqref{eq:alpha} and the two equations of
   \eqref{eq:beta} we get
    \begin{equation}
      \label{eq:4}
      \frac{d}{dt} \Big|_{t=0} h_t(v)=0
      \qquad \mathrm{and} \qquad
      \frac{d^2}{dt^2} \Big| _{t=0} h_t(v)=2 
      \left(
        \big[ X^h,Y^h \big] \! -\! \big[ X,Y \big]^h
      \right)_v = 2 \xi_p.
  \end{equation}
  where we also used \eqref{eq:R}.  To summarize, we get from \eqref{eq:h_0} and
  \eqref{eq:4}:
  \begin{equation}
    h_{0,p}= \mathrm{id}\big|_{\I_p}, \qquad\qquad 
    \frac{d}{dt} \Big|_{t=0} h_{t,p}= 0,
    \qquad\qquad \frac{1}{2} \frac{d^2}{dt^2}\Big|_{t=0} h_{t,p}= \xi_p,
  \end{equation}
  which means that the reparametrized map $t\to h_{t/\sqrt{2},p}$ is a
  second order integral curve of the curvature vector field
  $\xi_p\in \X{\I_p}$ and proves point \emph{(2)} of the lemma.
\end{proof}

\medskip

\begin{proof}[Proof of Proposition \ref{prop:gorb-erinto-uj}]
  Let us consider the map $h_t:\I M \to \I M$ on the indicatrix bundle, where
  $h_t\big|_{\I_p}:=h_{t,p}$. From Lemma \ref{lemma:h_t_p} we get (by dropping
  the variable $p\in M$) that
  \begin{itemize}[topsep=2pt, partopsep=0pt,leftmargin=34pt]
    \setlength{\itemsep}{2pt} \setlength{\parskip}{0pt}
    \setlength{\parsep}{0pt}
  \item [\emph{(1)}] $h_{t} \in \Holf$,
  \item [\emph{(2)}] $t\to h_{t}$ is a second order integral curve of the
    vector field $\xi \in \X{\I}$.
  \end{itemize}
  which shows that the curvature vector field $\xi$ is tangent to $\Holf$ and
  proves the first part of the proposition. Applying Corollary
  \ref{cor:sub_algebra}, we get that the Lie algebra generated by the
  curvature vector field is tangent to $\Holf$ which proves the second part of
  the proposition.
\end{proof}

We remark that \emph{(1)} of Proposition \ref{prop:gorb-erinto-uj} is an
improvement of Proposition 3.~and Corollary 2.~of
\cite{Muzsnay_Nagy_2011}. Indeed, the tangent property proved in
\cite{Muzsnay_Nagy_2011} is weaker: $C^1$ instead of $C^\infty$
smoothness. Moreover, \cite{Muzsnay_Nagy_2011} uses a very strong
topological restriction on the manifold $M$ supposing it is
diffeomorphic to the $n$-dimensional euclidean space. In Proposition
\eqref{prop:gorb-erinto-uj} we presented a natural geometric
construction without any constraints on the topology of the manifold
$M$.

\medskip


\begin{definition}
  The \textit{infinitesimal holonomy algebra} $\ihol (M)$ of a Finsler
  manifold $(M,\F)$ is the smallest Lie algebra on the indicatrix bundle which
  satisfies the following properties:
  \begin{itemize}[topsep=2pt, partopsep=0pt,leftmargin=30pt]
    \setlength{\itemsep}{2pt} \setlength{\parskip}{0pt}
    \setlength{\parsep}{0pt}
  \item[1)] Every curvature vector field $\xi$ is an element of $\ihol (M)$,
  \item[2)] if $\xi , \eta \in \ihol (M)$, then
    $\left[\xi, \eta \right]\in \ihol (M)$,
  \item[3)] if $\xi \in \ihol (M)$ and $X\in \X{M}$, then the horizontal
    Berwald covariant derivative $\nabla _X \xi$ is also an element of
    $\ihol (M)$.
  \end{itemize}
\end{definition}
We have the following 

\begin{proposition}
  \label{prop:inf.hol.erinto} \
  \begin{enumerate}[topsep=2pt, partopsep=0pt,leftmargin=30pt]
  \item The elements of the infinitesimal holonomy algebra $\ihol (M)$ are
    tangent to $\Holf$.
  \item The infinitesimal holonomy algebra $\ihol(M)$ is a Lie subalgebra of
    $\holf$.
  \end{enumerate}
\end{proposition}

\begin{proof}
  From Proposition \ref{prop:gorb-erinto-uj} we know that the curvature vector
  fields are tangent to the fibered holonomy group. Moreover, from
  \cite[Proposition 4]{Muzsnay_Nagy_2011} and from \emph{(i)} of Remark
  \ref{prop:weak_and_strong_tangent} we get that the horizontal Berwald
  covariant derivative of tangent vector fields to $\Holf$ are also tangent to
  $\Holf$ which proves the first part of the proposition. As a consequence,
  the infinitesimal holonomy algebra is generated by tangent vector fields
  and, according to Corollary \ref{cor:sub_algebra}, it is tangent to $\Holf$
  proving the second part of the proposition.
\end{proof}

We remark that the first part of Proposition \ref{prop:inf.hol.erinto} is an
improvement of \cite[Theorem 2]{Muzsnay_Nagy_2011}, because in Proposition
\ref{prop:inf.hol.erinto} the strong topology condition on the manifold $M$ is
dropped.

\bigskip \bigskip

\subsection{Holonomy algebra and its Lie subalgebras}
\label{sec:holonomy_algebra} \

\bigskip

\noindent
Let $(M, F)$ be an $n-$dimensional Finsler manifold.  At any points $p\in M$
the indicatrix defined in \eqref{eq:2} is an $(n-1)$-dimensional compact
manifold in $T_pM$.  Therefore, the diffeomorphism group $\diff{\I_p}$ is an
infinite dimensional Fréchet Lie group whose Lie algebra is $\X{\I_p}$, the
Lie algebra of smooth vector fields on $\I_p$. As it was introduced in Chapter
\ref{sec:finsl_hol}, the holonomy group
\begin{equation}
  \Hol_p(M) \subset \diff {\I_p M},  
\end{equation}
is a subgroup of the diffeomorphism group $\diff {\I_p M}$.  The set of
tangent vector fields to the group $\Hol_p(M)$, denoted as
\begin{displaymath}
  \hol_p(M):=\mathcal T_0 \big(\Hol_p(M) \big).
\end{displaymath}

\begin{definition}
  $\hol_p(M)$ is called the \emph{holonomy algebra} of the Finsler manifold
  $(M,\F)$ at $p\in M$.
\end{definition}
From Theorem \ref{thm:erinto_algebra} one can obtain
\begin{theorem}
  The holonomy algebra $\hol_p(M)$ of a Finsler manifold $(M,\F)$ at $p\in M$
  is a Lie subalgebra of $\X {\I_p}$.
\end{theorem}

\bigskip

In the sequel we identify two important Lie subalgebras of the holonomy
algebra of Finsler manifolds.
\begin{definition}
  A vector field $\xi_p\in \X{\I_p}$ on the indicatrix $\I_p\subset T_pM$ is
  called a curvature vector field at $p\in M$ if there exist tangent vectors
  $X_p,Y_p \in T_pM$ such that $\xi_p = R(X_p^h,Y_p^h)$.  The Lie subalgebra
  $\curv_p$ of vector fields generated by curvature vector fields at $p\in M$
  is called the \emph{curvature algebra at} $p$.
\end{definition}
The relationship between the curvature algebra $\curv_p$ at $p\in M$ and the
curvature algebra $\curv$ introduced in Definition \ref{def:curvature_algebra}
is:
\begin{displaymath}
  \curv_p= \left\{\,\xi_p=\xi|_{\I_p} \ \big | \ \xi \in \curv  \,\right\},
\end{displaymath}
that is $\curv_p$ is the restriction of $\curv$ to the indicatrix $\I_p$. We
have
\begin{proposition}
  \label{prop:gorb-erinto-uj_ind}\ The elements of the curvature algebra
  $\curv_p$ at $p\!\in\! M$ are tangent to the group $\Hol_p(M)$ and the
  curvature algebra $\curv_p$ is a Lie subalgebra of the holonomy algebra
  $\hol_p(M)$.
\end{proposition}
The proof is a direct consequence of the computation of Proposition
\eqref{prop:gorb-erinto-uj}. Moreover, by localizing the infinitesimal
holonomy algebra at a point we can obtain
\begin{definition}
  The Lie algebra
  \begin{math}
    \ihol_p(M) := \halmazvonal{ \xi|_{\I_p}}{ \xi \in \ihol(M)}
  \end{math}
  of vector fields on the indicatrix $\I_p$ is called the infinitesimal
  holonomy algebra at the point $p\in M$.
\end{definition}
From Proposition \ref{prop:inf.hol.erinto} we get
\begin{proposition}
  \label{prop:inf.hol.erinto_ind}
  The elements of the infinitesimal holonomy algebra $\ihol_p (M)$ are tangent
  to the group $\Hol_p(M)$ and the infinitesimal holonomy algebra $\ihol_p$ is
  a Lie subalgebra of the holonomy algebra $\hol_p(M)$.
\end{proposition}
We note that by the construction of the infinitesimal holonomy algebra, the
curvature vector fields are elements of $\ihol_p(M)$, therefore we have the
sequence of the Lie algebras
\begin{equation}
  \label{eq:lie_algebras}
  \curv_p(M) \ \subset \  \ihol_p(M) \ \subset \ \hol_p(M) \ \subset \ \X{\I_p}.
\end{equation}
We also remark that the first parts of the statement of Proposition
\ref{prop:gorb-erinto-uj_ind} and \ref{prop:inf.hol.erinto_ind} are
improvements of the results of \cite{Muzsnay_Nagy_2011} because the tangential
property of the Lie algebra is improved: we can guaranty $C^\infty$-smoothness
instead of $C^1$-smoothness.

\bigskip \bigskip

\noindent
\textbf{Concluding remarks.}  

\bigskip

\noindent
Many interesting geometric results can be obtained on the holonomy structure
from the Lie algebras \eqref{eq:lie_algebras} through the tangent property.
Indeed, by using Theorem \ref{thm:expo}, one can find examples where, in
contrast to the Riemannian case, the holonomy group $\Hol_p(M)$ is not a
compact Lie group \cite{Muzsnay_Nagy_2012_hjm, Muzsnay_Nagy_2012,
  Muzsnay_Nagy_2014}, or where the closure of the holonomy group is the
infinite dimensional Lie group $\mathcal{D}i\!f\!f^{\infty}_+(\I_p)$ of the
orientation preserving diffeomorphism group of the indicatrix
\cite{Hubicska_Muzsnay_2017,Muzsnay_Nagy_max_2015}.  All these results were
obtained by using the tangent property of the curvature algebra $\curv_p(M)$
and the infinitesimal holonomy algebra $\ihol_p(M)$. The method developed in
Chapter \ref{sec:tangen_algebra}, however, allows us to introduce in a natural
and canonical way a potentially larger Lie algebra, the holonomy algebra,
which is the \emph{tangent Lie algebra} of the holonomy group. This Lie
algebra gives the best linear approximation of the holonomy group. The
technique can be applied in other fields of geometry as well. We are convinced
that the method, exploring the tangential property of a group associated with
a geometric structure, can be used successfully to investigate various
geometric properties.

\bigskip

\textbf{Aknowledgement:} The authors would like to thank the referee for
the constructive comments and recommendations which contributed to
improving the paper.

\bigskip \bigskip


\begin{thebibliography}{10}


\bibitem{Berger_1955} M.~Berger.  \newblock Sur les groupes d'holonomie
  homog\`ene des vari\'et\'es \`a connexion affine et des vari\'et\'es
  riemanniennes.  \newblock {\em Bull. Soc. Math. France}, 83:279--330,
  1955.

\bibitem{Borel_Lichn_1952} A.~Borel and A.~Lichnerowicz.  \newblock
  Groupes d'holonomie des vari\'et\'es riemanniennes.  \newblock {\em
    C. R. Acad. Sci. Paris}, 234:1835--1837, 1952.

\bibitem{Bryant_2000} R.~Bryant.  \newblock Recent advances in the
  theory of holonomy.  \newblock {\em Ast\'erisque}, 266:Exp.\ No.\ 861,
  5, 351--374, 2000.  \newblock S\'eminaire Bourbaki, Vol. 1998/99.

\bibitem{Grifone_1972} J.~Grifone.  \newblock Structure presque-tangente
  et connexions. {I}.  \newblock {\em Ann. Inst. Fourier (Grenoble)},
  22(1):287--334, 1972.

\bibitem{Hubicska_Muzsnay_2017} B.~{Hubicska} and Z.~{Muzsnay}.
  \newblock {The holonomy groups of projectively flat Randers
    two-manifolds of constant curvature.}  \newblock {\em {preprint}},
  2017.

\bibitem{Joyce_2000} D.~D. Joyce.  \newblock {\em Compact manifolds with
    special holonomy}.  \newblock Oxford Mathematical Monographs. Oxford
  University Press, Oxford, 2000.

\bibitem{Kozma_2000} L.~Kozma.  \newblock On holonomy groups of
  {L}andsberg manifolds.  \newblock {\em Tensor (N.S.)}, 62(1):87--90,
  2000.

\bibitem{Mauhart_Michor_1992} M.~Mauhart and P.~W. Michor.  \newblock
  Commutators of flows and fields.  \newblock {\em Arch. Math. (Brno)},
  28(3-4):229--236, 1992.

\bibitem{Michor_1991} P.~W. Michor.  \newblock {\em Gauge theory for
    fiber bundles}, volume~19 of {\em Monographs and Textbooks in
    Physical Science. Lecture Notes}.  \newblock Bibliopolis, Naples,
  1991.

\bibitem{Muzsnay_Nagy_2011} Z.~Muzsnay and P.~T. Nagy.  \newblock
  Tangent {L}ie algebras to the holonomy group of a {F}insler manifold.
  \newblock {\em Commun. Math.}, 19(2):137--147, 2011.

\bibitem{Muzsnay_Nagy_2012_hjm} Z.~Muzsnay and P.~T. Nagy.  \newblock
  Finsler manifolds with non-{R}iemannian holonomy.  \newblock {\em
    Houston J. Math.}, 38(1):77--92, 2012.

\bibitem{Muzsnay_Nagy_2012} Z.~Muzsnay and P.~T. Nagy.  \newblock Witt
  algebra and the curvature of the {H}eisenberg group.  \newblock {\em
    Commun. Math.}, 20(1):33--40, 2012.

\bibitem{Muzsnay_Nagy_2014} Z.~Muzsnay and P.~T. Nagy.  \newblock
  Characterization of projective finsler manifolds of constant curvature
  having infinite dimensional holonomy group.  \newblock {\em
    Publ. Math. Debrecen}, 84(1-2):17--28, 2014.

\bibitem{Muzsnay_Nagy_max_2015} Z.~Muzsnay and P.~T. Nagy.  \newblock
  Finsler 2-manifolds with maximal holonomy group of infinite dimension.
  \newblock {\em Differential Geom. Appl.}, 39:1--9, 2015.

\bibitem{Omori_1997} H.~Omori.  \newblock {\em Infinite-dimensional
    {L}ie groups}, volume 158 of {\em Translations of Mathematical
    Monographs}.  \newblock American Mathematical Society, Providence,
  RI, 1997.  \newblock Translated from the 1979 Japanese original and
  revised by the author.

\bibitem{Spivak_1979} M.~Spivak.  \newblock {\em A comprehensive
    introduction to differential geometry. {V}ol.  {I}}.  \newblock
  Publish or Perish, Inc., Wilmington, Del., second edition, 1979.

\bibitem{Szabo_1981} Z.~I. Szab\'o.  \newblock Positive definite
  {B}erwald spaces. {S}tructure theorems on {B}erwald spaces.  \newblock
  {\em Tensor (N.S.)}, 35(1):25--39, 1981.

\bibitem{Szilasi_Lovas_Kertesz_2014} J.~{Szilasi}, R.~L. {Lovas}, and
  D.~C. {Kert\'esz}.  \newblock {\em {Connections, sprays and Finsler
      structures.}}  \newblock Hackensack, NJ: World Scientific, 2014.

\bibitem{Wojtynski_2003} W.~{Wojty\' nski}.  \newblock {\em {Groups of
      strings.}}  \newblock {\em J. Lie Theory} 13, no. 2, 359--382, 2003.

\end{thebibliography}
\end{document}